%% file: main.tex
\documentclass[centering,reqno]{amsart}  
\input{header}

\newcommand{\tempnewpage}{}

\newcommand{\dual}[1]{{#1}^{\circ}}
\newcommand{\fdual}[1]{{#1}^{\diamond}}

\newcommand{\sarticle}{a}

\newcommand{\sname}{special}
\newcommand{\Sname}{Special}
\newcommand{\sface}{\sname\ face}

\newcommand{\ansface}{\sarticle\ \sface}

\newcommand{\sfaces}{\sname\ faces}

\newcommand{\spoint}{\sname\ point}

\newcommand{\anspoint}{\sarticle\ \spoint}

\newcommand{\spoints}{\sname\ points}
\newcommand{\Spoints}{\Sname\ points}

\renewcommand{\mid}{:}

\newcommand{\claimstyle}[1]{\textbf{#1}}

\begin{document}

%%%%%%%%%%%%%%%%%%%%%%%%%%%%%%%%%%%%%%%%%%%%%%%%%%%%%%%%%%%%%%%%%%%%%%%%%%%%%%%%%%%%%%

\expandafter\title
[Kalai's $3^{d}$ conjecture for locally anti-blocking polytopes]
{Kalai's $\boldsymbol{3^{d}}$ conjecture for unconditional and locally anti-blocking polytopes}
		
\author[R.\ Sanyal]{Raman Sanyal}
\author[M.\ Winter]{Martin Winter}
\address{Institute for Mathemtics, Goethe University Frankfurt, Robert-Mayer-Str.\ 10, 60325 Frankfurt am Main, Germany}
\email{sanyal@math.uni-frankfurt.de}
\address{Mathematics Institute, University of Warwick, Coventry CV4 7AL, United Kingdom}
\email{martin.h.winter@warwick.ac.uk}
	
\subjclass[2010]{51M20, 52B05, 52B11, 52B12, 52B15}
% 51M20 - Polyhedra and polytopes
% 52B05 - Combinatorial properties of polytopes and polyhedra
% 52B11 - $n$-dimensional polytopes
% 52B12 - Special polytopes
% 52B15 - Symmetry properties of polytopes
% 52B20 - Lattice polytopes in convex geometry (including relations with commutative algebra and algebraic geometry)
% 52B60 - Isoperimetric problems for polytopes
% 05C50 - Graphs and linear algebra
% 05C62 - Graph representations
\keywords{Kalai's $3^d$ conjecture, Mahler conjecture, centrally-symmetric polytopes, unconditional polytopes, locally anti-blocking polytopes, Hanner polytopes}
		
\date{\today}
\begin{abstract}
Kalai's $3^d$ conjecture states that every centrally-symmetric $d$-poly\-tope has at least $3^d$ faces. We give short proofs for two special cases:\nls if~$P$~is unconditional (that is, invariant \wrt\ reflection in any coordinate hyper\-plane), and more generally, if $P$ is locally anti-blocking. % (that is, looks like~an unconditional polytope in every orthant). 
In both cases we show that the~min\-imum is attained exactly for the Hanner polytopes.
\end{abstract}

\maketitle

\input{sec/introduction}

\par\bigskip
\parindent 0pt
\textbf{Acknowledgements.} 
We thank Matthias Schymura for his valuable feedback on the article. 
We also thank the reviewer for helpful comments.
The second author gratefully acknowledges the support by the funding by the British Engineering and Physical
Sciences Research Council [EP/V009044/1].

%%%%%%%%%%%%%%%%%%%%%%%%%%%%%%%%%%%%%%%%%%%%%%%%%%%%%%%%%%%%%%%%%%%%%%%%%%%%%%%%%%

 \bibliographystyle{abbrv}
 \bibliography{literature}

\end{document}

%% file: header.tex
\usepackage[T1]{fontenc}
\usepackage[utf8]{inputenc}
\usepackage[USenglish]{babel}

\usepackage{mathabx}

\usepackage[
	bookmarks=true,
	plainpages=false,
	linktocpage,
	colorlinks=true,
	citecolor=green!80!black,
	linkcolor=red!70!black,
	filecolor=magenta,
	urlcolor=magenta,
	breaklinks,
	pdfauthor={Martin Winter},
]{hyperref}

\usepackage{amsmath,amsthm}
\usepackage{calc,mathtools} %calc for \widthof, mathtools for \mathclap

\iftrue
\usepackage{amssymb} % incompatible with mathdesign
\else
\usepackage[charter,cal=cmcal]{mathdesign}
\fi

\usepackage{colortbl,color} % colors
\usepackage[dvipsnames]{xcolor}
\usepackage{centernot} % crossing out symbols correctly
\usepackage{array} % arrays
\usepackage{enumitem,moreenum} % numbered lists where each item can be labeled
\usepackage{cite} % for grouping references in the text which are adjacent in the bibliography
\usepackage[nameinlink,capitalize,noabbrev]{cleveref}
\usepackage{nicefrac}

\usepackage{tikz-cd} 

\usepackage[font=small,labelfont=bf]{caption}

\usepackage{blkarray}

\usepackage{inconsolata}

\usepackage[
    colorinlistoftodos,
    backgroundcolor=yellow,
    textsize=footnotesize
]{todonotes}

\newcommand{\RR}{\mathbb{R}}    % real numbers                      
\newcommand{\F}{\mathcal{F}}    % face lattice  
                   
\newcommand{\Sph}{\mathbb{S}}    % sphere

\def\^#1{^{(#1)}}
\def\s^#1{^{\smash{(#1)}}}

\newcommand{\wideword}[1]{\quad\text{#1}\quad}
\newcommand{\wideand}{\wideword{and}}

\def\:{\colon}

\newcommand{\cupdot}{\mathbin{\mathaccent\cdot\cup}}

\newcommand{\labelstyle}[1]{\upshape(\textit{#1})}
\newcommand{\mylabel}{\labelstyle{\roman*}}

\newenvironment{myenumerate}
    {\begin{enumerate}[label=\mylabel]}
    {\end{enumerate}}

\def\itm#1{{\labelstyle{\romannumeral#1\relax}}}
\def\itmto#1#2{\itm#1\,$\Rightarrow$\,\itm#2}
\def\itmeq#1#2{\itm#1\,$\Leftrightarrow$\,\itm#2}

\def\nlspace{\nolinebreak\space}

\newcommand{\freespace}{\kern.07em} 
\newcommand{\free}{\freespace\cdot\freespace} 

\newcommand{\ul}[1]{\underline{\smash{#1}}}

\newcommand{\msays}[1]{{\footnotesize\textcolor{red}{\textbf{M:} #1}}}

\newcommand{\TODO}{\textcolor{red}{TODO}}

\newtheoremstyle{mythmstyle} % name
    {\parsep}                    % Space above
    {\parsep}                    % Space below
    {\itshape}                   % Body font
    {}                           % Indent amount
    {\bfseries\scshape}          % Theorem head font
    {.}                          % Punctuation after theorem head
    {.5em}                       % Space after theorem head
    {}  % Theorem head spec (can be left empty, meaning ‘normal’)
\newtheoremstyle{mydefstyle} % name
    {\parsep}                    % Space above
    {\parsep}                    % Space below
    {}                   % Body font
    {}                           % Indent amount
    {\mdseries\scshape}          % Theorem head font
    {.}                          % Punctuation after theorem head
    {.5em}                       % Space after theorem head
    {}  % Theorem head spec (can be left empty, meaning ‘normal’)

\numberwithin{equation}{section}

\theoremstyle{theorem}
\newtheorem{theorem}{Theorem}%[section]
\newtheorem{corollary}[theorem]{Corollary}
\newtheorem{lemma}[theorem]{Lemma}
\newtheorem{proposition}[theorem]{Pro\-po\-si\-tion}
\newtheorem{conjecture}[theorem]{Conjecture}

\newtheorem{challenge}[theorem]{Challenge}

\theoremstyle{definition}

\newtheorem{example}[theorem]{Example}
\newtheorem{remark}[theorem]{Remark}

\newtheorem{claimx}{Claim}

\crefname{theorem}{Theorem}{Theorems}
\crefname{proposition}{Proposition}{Propositions}
\crefname{lemma}{Lemma}{Lemmas}
\crefname{corollary}{Corollary}{Corollaries}
\crefname{remark}{Remark}{Remarks}
\crefname{example}{Example}{Examples}
\crefname{definition}{Definition}{Definitions}
\crefname{problem}{Problem}{Problems}
\crefname{observation}{Observation}{Observation}
\crefname{construction}{Construction}{Construction}

\DeclareMathOperator{\aff}{aff}
\DeclareMathOperator{\conv}{conv}
\DeclareMathOperator{\cone}{cone}

\DeclareMathOperator{\Span}{span}

\DeclareMathOperator*{\Argmax}{Argmax}   	
\DeclareMathOperator{\vol}{vol}  	
\DeclareMathOperator{\Int}{int}  
\DeclareMathOperator{\cl}{cl}  	
\DeclareMathOperator{\relint}{relint}

\DeclareMathOperator{\supp}{supp} 
\DeclareMathOperator{\sign}{sign} 
\DeclareMathOperator{\grad}{grad}

\let\eps=\epsilon
\let\eset=\varnothing

\let\<=\langle
\let\>=\rangle

\def\...{...}

\renewcommand{\precdot}{\prec\mathrel{\mkern-2mu}\mathrel{\cdot}}

\newcommand{\shortStyle}{\textit}
\newcommand{\ie}{\shortStyle{i.e.,}}

\newcommand{\eg}{\shortStyle{e.g.}}

\newcommand{\wrt}{\shortStyle{w.r.t.}}

\newcommand{\cf}{\shortStyle{cf.}}

\makeatletter
\renewcommand*{\eqref}[1]{%
  \hyperref[{#1}]{\textup{\tagform@{\ref*{#1}}}}%
}
\makeatother

\def\nlspace{\nolinebreak\space}
\def\nls{\nlspace}

%% file: sec/introduction.tex
\newcommand\Def[1]{\textbf{#1}}%

% \msays{I use \textbackslash msays\{...\} for comments attributed to me (will be shown in red).}
% \todo{\msays\ Use \textbackslash todo\{...\} for comments that need not be inline}

% \ssays{You can use \textbackslash ssays\{...\} for your comment (will be shown in blue)}
% \todo[color=lightgray]{\msays\ I prefer setting comments to gray instead of deleting them once addressed if they were not by myself. They can then be deleted by the author once they agree that the comment has been addressed.}

% \msays{For bigger changes in, say, a proof, I suggest you copy the proof so that we have both version in the document and can talk about them.}

\section{Introduction}
\newcommand\PMZ{\{-1,0,+1\}}%

% Where to cite this properly?
\nocite{ChambersPortnoy}

A convex polytope $P \subset \RR^d$ is \Def{centrally-symmetric} if $-P=P$.
Central~symmetry strongly affects the combinatorics. For example,
Figiel--Lindenstrass--Milman \cite{FLM} proved that $P$ cannot simultaneously
have few vertices and few facets. However, our understanding of extremal cases
is severely limited. This is best illustrated by the following basic conjecture
due to Kalai~\cite{Kalai3d}. We write $s(P)$ for~the~\mbox{number~of~non}-empty faces
of $P$ and we denote by $C_d = [-1,+1]^d$ the $d$-dimensional cube.

\begin{conjecture}[$3^d$ conjecture]\label{conj:kalai}
    Every $d$-dimensional centrally-symmetric polytope $P$ satisfies
    \[
        s(P) \ \ge \ s(C_d) \ = \ 3^d \, .
    \]
\end{conjecture}

The conjecture is known to hold in dimensions $d \le 4$;
see~\cite{SanyalWernerZiegler}. The $3^d$ conjecture is likened to the
famous Mahler conjecture~\cite{Mahler} in convex geometry;
\cf\ \cite{Schneider-book}. We~write
$P^\circ$ for the polar dual of $P$ and $M(P) := \vol_d(P) \vol_d(P^\circ)$
for the \Def{Mahler volume}.

\begin{conjecture}[Mahler conjecture]\label{conj:kalai}
    Every $d$-dimensional centrally-symmetric polytope $P$ satisfies
    \[
        M(P) \ \ge \ M(C_d) \ = \ \frac{4^d}{d!} \, .
    \]
\end{conjecture}

Intuitively, the Mahler volume can be thought of as a
($\mathrm{GL}_d$-invariant) measure of roundness and, by the Blaschke--Santaló inequality, is maximized on the Euclidean ball. The conjecture states that cubes are the ``least round'' centrally-symmetric polytopes. 
Round polytopes require many faces and $s(P)$ should be thought of as a combinatorial measure of roundness and hence should also be minimized on the cube. 
It turns out that in both cases the conjectured minimizers are the \Def{Hanner polytopes}: every $1$-dimensional centrally-symmetric polytope is a Hanner polytope; and a polytope of dimension $d \ge 2$ is a Hanner polytope if $P$ or $P^\circ$ is linearly isomorphic to the Cartesian product of two lower-dimensional Hanner
polytopes.

A full-dimensional polytope $P \subset \RR^d$ is \Def{unconditional} (or coordinate-symmetric) if $P$ is
symmetric with respect to reflection in every coordinate hyperplane.  The
Mahler conjecture holds for unconditional polytopes by the work of
Saint-Raymond \cite{Saint-Raymond}. 
Hanner polytopes are unconditional, and
Reisner~\cite{Reisner-unconditional} showed that these are~the only minimizers among all
unconditional polytopes. 

The following main result of our paper is an analogue for the $3^d$ conjecture:

\begin{theorem}\label{cor:main_unconditional}
    If $P \subset \RR^d$ is an unconditional polytope, then $s(P) \ge 3^d$.
    Moreover, $s(P) = 3^d$ if and only if $P$ is a Hanner polytope.
\end{theorem}

This is a consequence of a stronger result. We need the following terminology:\nls for $J \subseteq [d]:=\{1,...,d\}$ 
we define the \Def{coordinate subspace}~\mbox{$\RR^d_J:=\Span\{e_i\mid i\in J\}$}, where~$e_i$ is the $i$-th unit basis vector.
%let $\RR^d_J$ be the coordinate subspace~of~points $x \in \RR^d$ with $x_i = 0$ whenever $i \not \in J$.  
Further let $\pi_J : \RR^d \to \RR^d_J$ be the
associated orthogonal projection (we write $\pi_i$ if $J=\{i\}$). 
A polytope $P \subset \RR^d$ is said~to~be
\Def{locally~anti-blocking} if $\pi_J(P) = P \cap \RR^d_J$ for all $J\subseteq[d]$.
Locally anti-blocking polytopes were introduced in \cite{KohlOlsenSanyal} and generalize unconditional polytopes (see \cref{fig:unconditional_anti-blocking} for further illustration of these concepts).  
\begin{figure}[h!]
    \centering
    \includegraphics[width=0.83\textwidth]{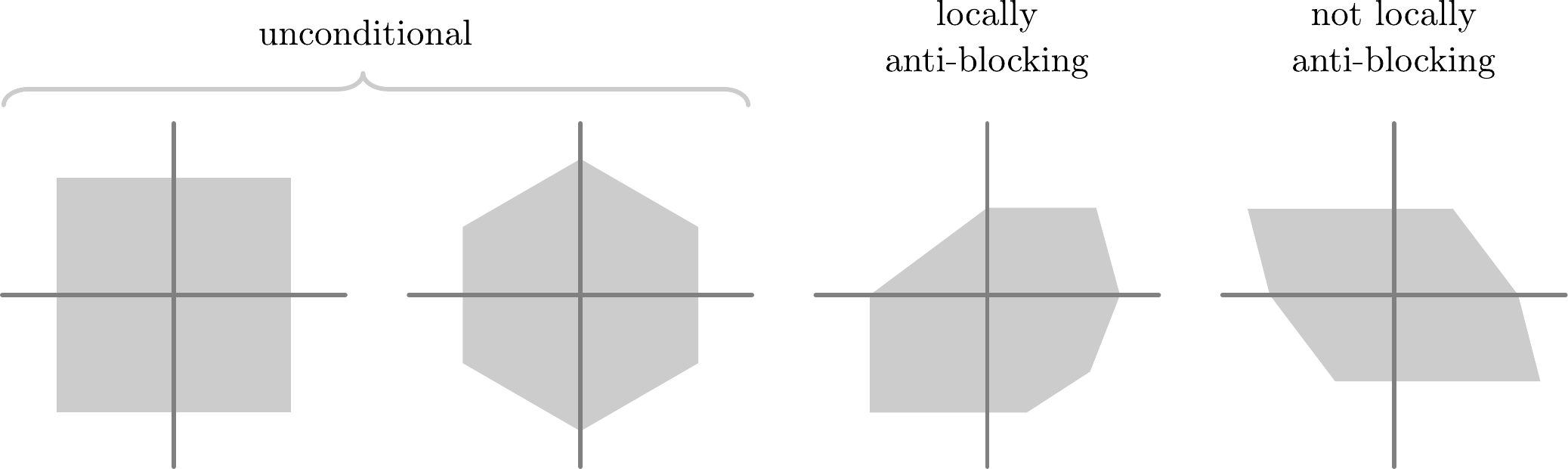}
    \caption{Visualization of two unconditional polytopes, a locally anti-blocking polytope that is not unconditional, and a polytope that is neither even though it is centrally-symmetric.}
    \label{fig:unconditional_anti-blocking}
\end{figure}

While locally anti-blocking polytopes are not necessarily centrally-symmetric,
it was shown in \cite{ArtsteinSadovskySanyal} that they still satisfy the Mahler
conjecture, including the~conjectured equality cases (at least locally in each orthant). 

We call a locally anti-blocking polytope \Def{proper} if the origin is contained in its interior.
\cref{cor:main_unconditional} is then a consequence of the following more general result:

\begin{theorem}\label{thm:main_lab}
    If $P \subset \RR^d$ is a proper locally anti-blocking polytope, then $s(P) \ge 3^d$. Moreover, $s(P) = 3^d$ if and only if $P$ is a (generalized) Hanner polytope.
\end{theorem}

Here, a \Def{generalized Hanner polytope} is any polytope that is combinatorially equivalent to a usual Hanner polytope and is locally anti-blocking (but not necessarily centrally-symmetric). 
In particular, generalized Hanner polytopes follow a recursive definition analogous to Hanner polytopes (see \cref{sec:minimizers}). % quite analogous to recursively exactly like Hanner polytope, except that for $d=1$ we do not require the line segments to be centrally-symmetric, but only to contain $0$ in their interior.

%Here, \Def{generalized Hanner polytopes} are combinatorially equivalent to Hanner polytopes and are defined recursively exactly like Hanner polytope, except that for $d=1$ we do not require the line segments to be centrally-symmetric, but only to contain $0$ in their interior.

% Here, \Def{generalized Hanner polytopes} are defined recursively exactly like Hanner polytope, except that for $d=1$ we do not require a centrally-symmetric line segment, but rather just any line segment that contains $0$ in its interior.
% In particular, a generalized Hanner polytope is combinatorially equivalent to a Hanner polytope.

To prove \Cref{thm:main_lab}, we show that there exist exactly $3^d$ so-called \emph{\spoints} in $P$ and that each face contains at most one of them in its relative interior. 
More precisely, for each $\sigma \in \PMZ^d$, we define a concave
function $f_\sigma$, whose maximum over $P$ is attained exactly at one of the \spoints.
While our approach focuses on the combinatorics of
locally anti-blocking polytopes, it is quite similar to the original proof of
Mahler's conjecture for unconditional
polytopes~\cite{Saint-Raymond,Reisner-unconditional} and maybe suggests a
stronger relation between the two conjectures.

We also give a second short and more combinatorial proof of the inequality~part of \cref{cor:main_unconditional} that pertains to the structure of face lattices of unconditional polytopes and that we consider to be of independent interest.
Another proof of~the inequality part for unconditional polytopes due to Chambers and Portnoy has simultaneously appeared in \cite{ChambersPortnoy}.

%While this proof does not settle the equality cases, it should be of independent interest.

\begin{remark}[Zonotopes]
    Another class of polytopes for which both the Mahler conjecture and the $3^d$ conjecture is known to hold are zonotopes. 
    The Mahler~conjecture for zonotopes was verified in \cite{Reisner-zonoids}.
    The $3^d$ conjecture for zonotopes is trivially true.
    Indeed, the non-empty faces of a zonotope are
    in one-to-one correspondence to the cones of the associated arrangement of
    linear hyperplanes. Now any choice of $d$ linearly independent hyperplanes
    already gives rise to $3^d$ cones. This also shows that the cube $C_d$ is the
    unique $d$-dimensional zonotope with $3^d$ non-empty faces.
\end{remark}

\begin{remark}[Inductive approach]
    If $P \subset \RR^d$ is unconditional and $H := \{x : x_i = 0\}$ is a
    coordinate hyperplane, then $P' := P \cap H$ is unconditional in $H \cong \RR^{d-1}$.\nls
    A very appealing approach would be to show that $s(P) \ge 3
    s(P')$ and then to prove the~$3^d$ conjecture by induction. This, however,
    does not work, as sections can become ``too large'', already for unconditional polytopes. 
    The $3$-dimensional polytope with $8$ vertices given by $(\pm1,\pm1,0)$, $(\pm2,0,\pm1)$ has $f$-vector $(8,14,8)$ and hence $s(P) = 31$. The intersection with $H$ however is a hexagon with $s(P') = 13$; see \Cref{fig:counterexample}.
    \begin{figure}[h!]
    \centering
    \includegraphics[width=0.65\textwidth]{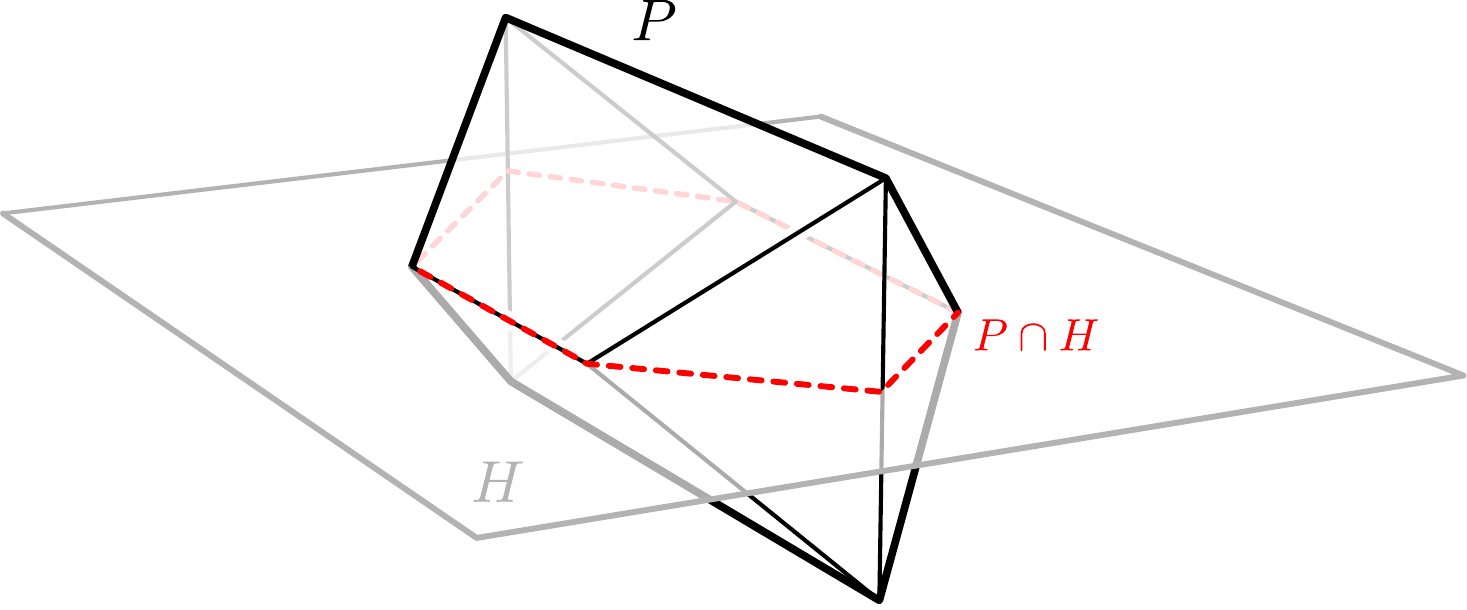}
    \caption{An unconditional polytope $P$ and a section $P\cap H$ (red) with a coordinate hyperplane that ``exposes'' more than a third of the faces.}
    \label{fig:counterexample}
\end{figure}
\end{remark}

\subsection{Overview} 
In \cref{sec:reminder_lab} we recall the fundamental properties of locally anti\-blocking polytopes that are most relevant to our approach.
The short proof of the inequality part of \cref{thm:main_lab} is given in \cref{sec:inequality}. Characterizing the minimizers requires more work and takes place in \cref{sec:minimizers}.
In \Cref{sec:uncondition}, we give the second combinatorial proof of \cref{cor:main_unconditional}.

\newcommand\inner[1]{\langle #1 \rangle}%
\newcommand\inv[1]{\widecheck{#1}}%

\tempnewpage

\section{Basic facts about locally anti-blocking polytopes}
\label{sec:reminder_lab}

A polytope $P\subset\RR^d$ is \Def{locally anti-blocking} if $\pi_J(x)\in P$ holds for all $x\in P$ and $J\subseteq[d]$.
In \cref{res:reminder_lab} below we collect some well-known elementary consequences of this definition.

Recall that $N_P(x)$ denotes the normal cone of $P$ at a point $x\in P$, \ie\ $$N_P(x):=\{v\in\RR^d\mid \<y-x,v\>\le 0\text{ for all $y\in P$}\}.$$
We also write $N_P(F)$ for the normal cone at an interior point of the face $F\subseteq P$.

\begin{samepage}
\begin{lemma}[see \cite{ArtsteinSadovskySanyal}]
\label{res:reminder_lab}
Let $P$ be a locally anti-blocking polytope
\begin{myenumerate}
    \item Sections with and projections onto coordinate subspaces result in the same polytope $P\cap\RR^d_J=\pi_J(P)=: P_J$. $P_J$  is itself locally anti-blocking.
    \item The polar dual $\dual P$ is again locally anti-blocking.
    %
    % \begin{align*}
    %     v\in \dual P 
    %     &\;\Longrightarrow\; \<x,v\> \le 1 \text{ for all $x\in P$}
    %     \\
    %     &\;\Longrightarrow\; \<\pi_J(x),v\> \le 1 \text{ for all $x\in P$}
    %     \\
    %     &\;\Longrightarrow\; \<x,\pi_J(v)\> \le 1 \text{ for all $x\in P$}
    %     \\
    %     &\;\Longrightarrow\; \pi_J(v)\in \dual P.
    % \end{align*}
    %\item 
    Moreover, \mbox{polar~duality~com}\-mutes with sections and~projections: $(\dual P)_J =(P_J)^\circ \cap\RR^d_J=:P_J^\circ$.
    
    \item For $p\in P_J$ holds 
    %$N_P(p) \subseteq N_{P_J}(p)$ and 
    $N_P(p)\cap \RR^d_J= N_{P_J}(p)\cap\RR^d_J$.
    % \begin{align*}
    %     v\in N_P(p)\cap\RR^d_J
    %     &\;\Longleftrightarrow\; \text{$v\in\RR^d_J$ and $\<v,p-x\>$ for all $x\in P$}
    %     \\
    %     &\;\Longleftrightarrow\;  \text{$v\in\RR^d_J$ and $\<\pi_J(v),p-x\>$ for all $x\in P$}
    %     \\
    %     &\;\Longleftrightarrow\;  \text{$v\in\RR^d_J$ and $\<v,\pi_J(p-x)\>$ for all $x\in P$}
    %     \\
    %     &\;\Longleftrightarrow\;  \text{$v\in\RR^d_J$ and $\<v,p-x\>$ for all $x\in P_J$}
    %     \\
    %     &\;\Longleftrightarrow\; v\in N_{P_J}(p)\cap\RR^d_J.
    % \end{align*}
\end{myenumerate}    
\end{lemma}
\end{samepage}

% \begin{proposition}\label{res:point_on_both_sides}
%     Let $F\subseteq P$ be a face and $x,y\in \relint(F)$.
%     If $\sign(x_i)\not=\sign(y_i)$ for some $i\in[d]$, then $v_i=0$ for all $v\in N_P(F)$.
% \end{proposition}

% In words: if a face $F$ contains points from both sides of a coordinate \mbox{hyperplane}, then its normal cone $N_P(F)$ resides entirely inside this hyperplane.

% \begin{proof}
%     Since $x,y\in\relint(F)$, we have $x':=\lambda x + (1-\lambda) y\in\relint(F)$ for some $\lambda>1$.
%     Also $x_i'<0$. 
%     Suppose that $v\in N_P(F)$ with $v_i\not=0$, say, $v_i>0$. By the~anti-blocking property, $\pi_i(x')=x'-x'_i e_i\in P$, and thus
%     %
%     $$\<\pi_i(x'),v\> = \<x',v\>-x'_i v_i > \<x',v\>$$
%     %
%     in contradiction to $x'\in F$ and $v\in N_P(F)$.
% \end{proof}

The following technical property of locally anti-blocking polytopes will be used repeatedly. Intuitively it states: if the interior of a face $F$ is not entirely contained on one side of a coordinate \mbox{hyperplane}, then its normal cone $N_P(F)$ must~be~contained in this hyper\-plane.

\begin{proposition}\label{res:point_on_both_sides}
    If $F\subseteq P$ is a face and $x,y\in \relint(F)$ are points with $\sign(x_i)\not=\sign(y_i)$ for some $i\in[d]$, then $v_i=0$ for all $v\in N_P(F)$.
\end{proposition}

\begin{proof}
   Assume w.l.o.g.\ $x_i\le 0 < y_i$.
    Since $x,y\in\relint(F)$, there is a $\lambda>1$~with $x':=\lambda x + (1-\lambda) y\in\relint(F)$ and $x_i'<0$.
    Suppose that~there~exists a $v\in N_P(F)$ with $v_i\not=0$, w.l.o.g.\ $v_i>0$. By the~anti-blocking property of $P$ we have $\pi_i(x')=x'-x'_i e_i\in P$, and thus
    $$\<\pi_i(x'),v\> = \<x',v\>-x'_i v_i > \<x',v\>,$$
    in contradiction to $x'$ being a maximizer of $\<\free,v\>$. Thus $v_i=0$.
\end{proof}

% \begin{proposition}\label{res:point_on_both_sides}
%     If $F\subseteq P$ is a face, $x,y\in \relint(F)$ are points and 
%     $J:=\{i\in[d]\mid \sign(x_i)\not=\sign(y_i)\}$, then $N_P(F)\subseteq\RR^d_J$.
%     %
% \end{proposition}

% \begin{proof}
%     Fix $i\in J$ and w.l.o.g.\ assume $x_i\le 0 < y_i$.
%     Since $x,y\in\relint(F)$, there is a $\lambda>2$ with $x':=\lambda x + (1-\lambda) y\in\relint(F)$ and $x_i'<0$.
%     Suppose that~there~exists a $v\in N_P(F)$ with $v_i\not=0$, say, $v_i>0$. By the~anti-blocking property of $P$ we have $\pi_i(x')=x'-x'_i e_i\in P$, and thus
%     %
%     $$\<\pi_i(x'),v\> = \<x',v\>-x'_i v_i > \<x',v\>,$$
%     %
%     in contradiction to $x'\in F$ and $v\in N_P(F)$. Thus $v_i=0$.
% \end{proof}

\tempnewpage

\section{The inequality part of \cref{thm:main_lab}}
\label{sec:inequality}

For this section let $P\subset\RR^d$ be a proper locally anti-blocking polytope. We show that $s(P)\ge 3^d$, proving the inequality part of \cref{thm:main_lab}.

For a point $p\in\RR^d$ we define $\inv{p}\in\RR^d$ by
$$\inv{p}_i:=\begin{cases}
    p_i^{-1} & \text{if $p_i\not=0$}
    \\
    0 & \text{otherwise}
\end{cases}.$$
Note that $\sign(p)=\sign(\inv p)$.
We say that a point $p\in P$ is \Def{\sname} if $\inv p\in N_P(p)$.
%A face $F\subseteq P$ is \Def{special} if there is a point $p\in F$ (a \Def{special point}) for which $\inv{p}$ is a normal vector to $P$ at $F$. For later we note that $\<p,\inv p\>=|\supp(p)|$.
%
%Our strategy for proving $s(P)\ge 3^d$ is to show that $P$ has exactly $3^d=|\PMZ|^d$ different \spoints, one per coordinate cone $\RR^d_\sigma$, and that each face of $P$ can contain at most one of them in its relative interior.
Our strategy for proving $s(P)\ge 3^d$ is to show that $P$ has exactly $3^d$ \spoints\ and that each face of $P$ can contain at most one of them in its relative interior.

For $\sigma\in\PMZ^d$ define the \Def{coordinate cone} $\RR^d_\sigma:=\{x\in\RR^d \mid \sign(x)=\sigma\}$. Those~are~relatively open cones with $\biguplus_\sigma\RR^d_\sigma=\RR^d$.

\begin{lemma}\label{res:3d_special_points}
    There are exactly $3^d$ \spoints\ in $P$.
    More precisely, for each~$\sigma\in\PMZ^d$ exists exactly one \spoint\ $p_\sigma\in \RR^d_\sigma$.
\end{lemma}
\begin{proof}
Define $f_\sigma\:\RR^d_\sigma\to\RR$ by
$$f_\sigma(x)=\sum_{\mathclap{i\in\supp(\sigma)}}\,\log(\sigma_i x_i).$$
Then $f_\sigma$ is strictly concave in $\RR^d_\sigma$ and has a unique maximizer on $P_\sigma:=P\cap\RR^d_\sigma$ (where $P_\sigma\not=\eset$ because $P$ is proper).
Note also that $\nabla f_\sigma(p)=\inv{p}$ for all~$p\in\RR^d_\sigma$.\nls
Below we show that $p_\sigma\in P_\sigma$ is \anspoint\ of $P$ if and only if it is~the~unique~maximizer of $f_\sigma$ in $P_\sigma$.
Since $\biguplus_\sigma P_\sigma= P$, this proves the claim.
%This proves the claim.

A point $p_\sigma\in P_\sigma$ maximizes $f_\sigma$ if and only if for all $q\in P_\sigma$ holds $\<q-p_\sigma,\nabla f_\sigma(p_\sigma)\>$ $=\<q-p_\sigma,\inv p_\sigma\>\le 0$, which is the definition of $\inv p_\sigma\in N_{P_\sigma}(p_\sigma)\cap\RR^d_\sigma=N_P(p_\sigma)\cap\RR^d_\sigma$. We used the anti-blocking property in form of \cref{res:reminder_lab} \itm3. % for the equality of normal cones.
% $\inv{p}\in N_{P_\sigma}(p)\cap\RR^d_\sigma = N_P(p)\cap\RR^d_\sigma$, where the equality is due to $P$ being locally anti-blocking (see \cref{sec:reminder_lab} \itm 4).
%
%If $p$ maximizes $f_\sigma$ in $P_\sigma$, then $\<p-q,\nabla f_\sigma(p)\>=\<p-q,\inv p\>\le 0$ for all $q\in P_\sigma$.
%This is exactly the definition of $\inv p\in N_{P_\sigma}(p)$.
%Since $\sigma=\sign(\inv p)$ we have $\inv p\in N_{P_\sigma}(p)\cap\RR^d_\sigma=N_P(p)\cap\RR^d_\sigma$ by \cref{sec:reminder_lab} \itm 4.
%
% The claim now follows from the following statement: $p\in P$ is \anspoint\ if and only if it is the unique maximizer of $f_\sigma$ in $P_\sigma$ for some $\sigma\in\PMZ^d$ (in~fact, $\sigma=\sign(p)$).
%
% A point $p\in P_\sigma$ maximizes $f_\sigma$ if and only if for all $q\in P_\sigma$ holds $\<q-p,\nabla f_\sigma(p)\>$ $=\<q-p,\inv p\>\le 0$, which is exactly the definition of $\inv{p}\in N_{P_\sigma}(p)\cap\RR^d_\sigma = N_P(p)\cap\RR^d_\sigma$, where the latter equality follows from $P$ being locally anti-blocking.
\end{proof}

We continue to write $p_\sigma$ for the uniquely determined \spoint\ of $P$~in~$\RR^d_\sigma$.\nls
The inequality part of \cref{thm:main_lab} is now an immediate consequence of the following.

\begin{lemma}\label{res:one_per_face}
    A face of $P$ contains at most one special point in its relative interior.
\end{lemma}

\begin{proof}
    Suppose $F\subseteq P$ is a face of $P$ with $p_\sigma,p_\tau\in\relint(F)$. 
    Then $\inv p_\sigma,\inv p_\tau\in N_P(F)$.
    If $\sign((p_\sigma)_i)=\sigma_i\not=\tau_i=\sign((p_\tau)_i)$ for some $i\in[d]$, then by \cref{res:point_on_both_sides}~$(\inv p_\sigma)_i=(\inv p_\tau)_i=0$ $\Rightarrow$ $\sigma_i=0=\tau_i$, which is a contradiction. Thus $\sigma=\tau$.
    %
    % Suppose $F\subseteq P$ is a face and $p_\sigma,p_\tau\in\relint(F)$.
    % Note that $\inv p_\tau\in N_P(F)$.
    % If $\sigma\not=\tau$, then w.l.o.g.\ $(p_\sigma)_i\le 0 < (p_\tau)_i$ for some $i\in[d]$. Choose $\lambda>1$ small enough so that $r:=\lambda p_\sigma+(1-\lambda) p_\tau\in \relint(F)$. 
    % In particular, $r_i<0$ and $\inv p_\tau\in N_P(r)$.\nls
    % By the anti-blocking property we have $r' := \pi_i(r) = r - r_ie_i \in P$.
    % We compute~$\inner{\inv p_\tau,r'} = \inner{\inv p_\tau,r} - r_i(p_\tau)_i > \<\inv p_\tau,r\>$, in contradiction to $\inv p_\tau\in N_P(r)$.
    % Thus, $\sigma=\tau$.
\end{proof}

We shall call a face \Def{\sname} if it contains \anspoint\ in its relative interior. By \cref{res:one_per_face}, there are exactly $3^d$ \sfaces\ which we denote \mbox{by~$F_\sigma,\sigma\in\PMZ^d$}, where $p_\sigma\in\relint(F_\sigma)$ and $\inv p_\sigma\in N_P(F_\sigma)$.

\tempnewpage

\section{The minimizer part of \cref{thm:main_lab}}
\label{sec:minimizers}

For this section we assume that $P\subset\RR^d$ is a proper locally anti-blocking polytope with exactly $s(P)=3^d$ non-empty faces.
Then each face of $P$ is \ansface\ $F_\sigma$ for some $\sigma\in\PMZ^d$.
Our goal is to show that $P$ is a generalized Hanner~polytope.
Recall that \emph{generalized Hanner polytopes} are locally anti-blocking realizations of~Hanner polytopes, and can be defined recursively as follows: for~$d=1$ a generalized Hanner polytope is a line segment that contains the origin in its interior. 
For $d\ge 2$ it is either the Cartesian product or the direct sum of two generalized Hanner polytopes of lower dimension, where by the direct sum of $P_1$ and $P_2$ we mean
$$P_1\oplus P_2:=\conv\{P_1\times\{0\}\cup \{0\}\times P_2\}.$$

%By cref{res:3d_special_points} and \cref{res:one_per_face}, each face of $P$ contains \emph{exactly one} \spoint\ $p_\sigma$ in its~relative interior, and we can enumerate the faces of $P$ as $F_\sigma\subseteq P,\sigma\in\PMZ^d$ so that $p_\sigma\in\relint(F_\sigma)$ and $\inv p_\sigma\in N_P(F_\sigma)$.

For $J\subseteq[d]$ we write $P_J:=P\cap\RR^d_J$, which is itself proper locally anti-blocking~of dimension $|J|$. In fact, it is a minimizer:

\begin{proposition}\label{res:3_to_J}
    For $J\subseteq[d]$ we have $s(P_J)=3^{|J|}$.
\end{proposition}

\begin{proof}
    We have $s(P_J)\ge 3^{|J|}$ by \cref{thm:main_lab}, and it suffices to show $s(P_J)\le 3^{|J|}$.

    For each face $F$ of $P_J$ there exists a unique minimal face $F_\sigma$ of $P$ so that~$F=F_\sigma\cap\RR^d_J$.
    We have $\inv p_\sigma\in N_P(F_\sigma)\cap\RR^d_\sigma=N_{P_J}(F)\cap\RR^d_\sigma$ by \cref{res:reminder_lab} \itm3, and hence $\supp(\sigma)\subseteq J$. 
    There are exactly $3^{|J|}$ sign-vectors with support in $J$.
    %
    % Then either $F_\sigma\subset\RR^d_J$, or $\relint(F)\subset\relint(F_\sigma)$.
    % In either case we can show $\supp(\sigma)\subseteq J$, which then proves the claim: in the first case this follows from $p_\sigma\in F_\sigma\subset \RR^d_J$.
    % In the second case it follows from $\inv p_\sigma\in N_P(F_\sigma)\subseteq N_{P_J}(F)\cap\RR^d_J$, where we used the anti-blocking property for the inclusion of the normal cones.
\end{proof}

Applying \cref{res:3_to_J} to the case $|J|=2$, we find $s(P_J)=9$ and so $P_J$ must be a quadrilateral. 
Moreover, in order to be locally anti-blocking it must be either a Cartesian product $[a_1,b_1]\times[a_2,b_2]$ (we say, it is \Def{axis-aligned}) or a direct sum $[a_1,b_1]\oplus[a_2,b_2]$~(we say, it is a \Def{diamond}); see \cref{fig:4_gon_cases}.

% Suppose that $|J|=2$, then from \cref{res:3_to_J} we conclude that $P\cap \RR^d_J$~is~a~quad\-rilateral and one can verify that they fall into one of the two classes depicted~in~\cref{fig:4_gon_cases}:
% \emph{axis-aligned} (if of the for $[a_1,b_1]\times[a_2,b_2]$) or \emph{not axis-aligned} (if of the~form $[a_1,b_1]\oplus[a_2,b_2]:=\conv([a_1,b_1]\times\{0\}\cup\{0\}\times[a_2,b_2])$).

\begin{figure}[h!]
    \centering
    \includegraphics[width=0.44\textwidth]{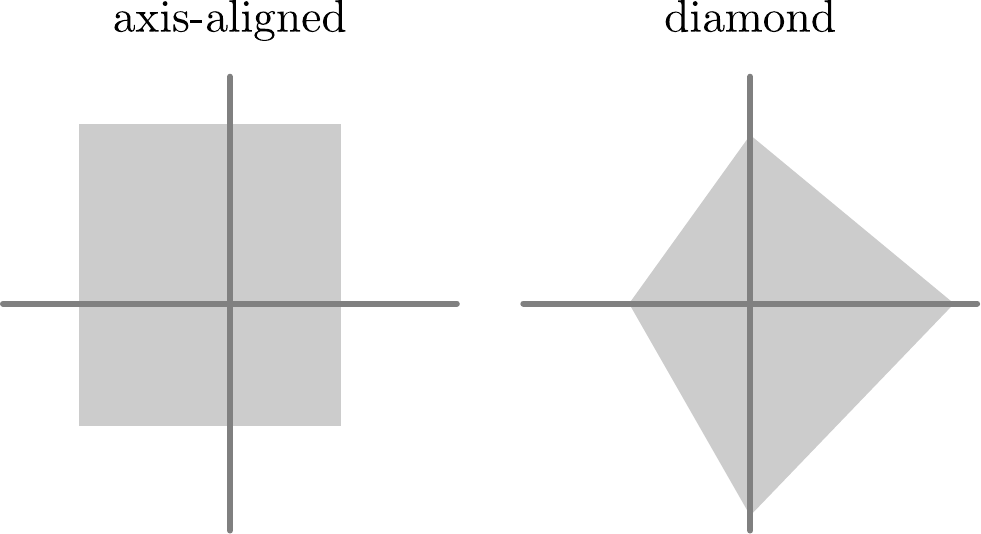}
    \caption{An axis-aligned quadrangle (left) and a diamond (right).}
    \label{fig:4_gon_cases}
\end{figure}

Based on this observation we build the graph $G_P$ on the vertex set $[d]$ with~$J=\{i,j\}$ forming an edge if and only if $P_J$ is axis-aligned. 
Some examples of this~construction are shown in \cref{fig:GP_examples}.
We will show that the graph $G_P$ captures the full combinatorial information about the minimizer $P$, and allows us to identify it as a generalized Hanner polytope.

% \begin{example}\quad
%     \begin{myenumerate}
%         \item If $P$ is an $d$-dimensional axis-aligned cube $[0,1]^d$, then $G_P$ is a complete graph on $d$ vertices.
%         \item If $P$ is a $d$-dimensional crosspolytope $\conv\{\pm e_i\mid i\in[d]\}$, then $G_P$ consists of $d$ isolated vertices.
%     \end{myenumerate}
% \end{example}

\begin{figure}
    \centering
    \includegraphics[width=0.9\textwidth]{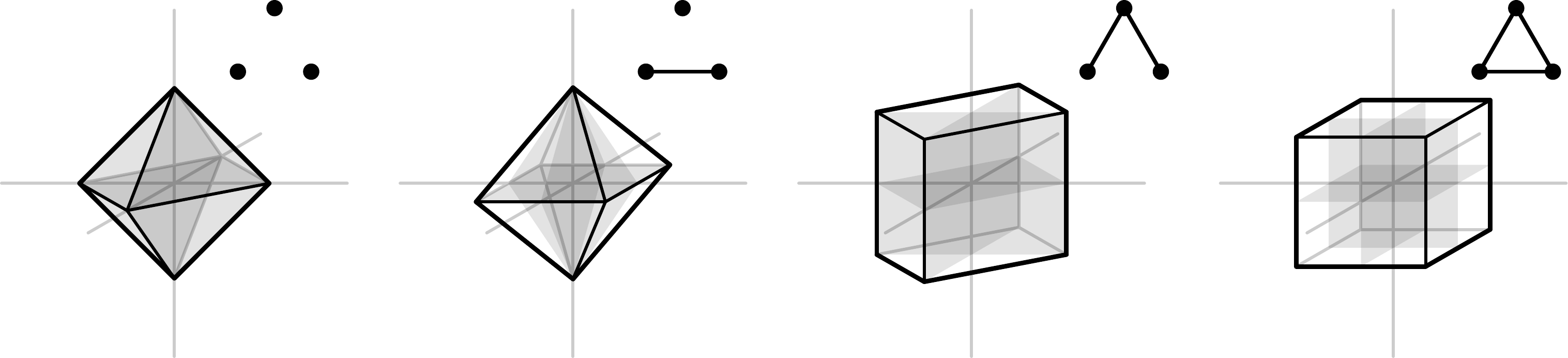}
    \caption{The 3-dimensional minimizers $P$ and their associa\-ted graphs $G_P$. 
    The shaded areas visualize intersections of $P$ with the 2-dimensio\-nal coordinate subspaces.
    It holds the more general observation that for an axis-aligned cube the graph $G_P$ is complete, and for an ``axis-aligned'' crosspolytope the graph $G_P$ is edge-less.}
    \label{fig:GP_examples}
\end{figure}

We collect some properties of $G_P$:

\begin{proposition}\label{res:GP_facts}\quad
\begin{myenumerate}
    \item $G_{\dual P}$ is the complement graph of $G_P$.
    \item $G_{P_J}$ is the subgraph of $G_P$ induced on $J\subseteq[d]$.
    \item $G_{P_1\oplus P_2}$ is the disjoint union $G_{P_1}\uplus G_{P_2}$.
\end{myenumerate}    
\end{proposition}
\begin{proof}
\emph{\itm1:}
This follows from $\dual P\cap\RR^d_J=P_J^\circ\cap\RR^d_J$ (using \cref{res:reminder_lab}~\itm2) and that the polar dual of an axis-aligned quadrangle is a diamond and vice versa.

\emph{\itm2:} This follows from $P_J\cap\RR^d_K=P_K$ whenever $K\subseteq J$.

\emph{\itm3:}
Consider $P_1\subset\RR^{d_1}\times\{0\}$ and $P_2\subset\{0\}\times \RR^{d_2}$. 
If we have $J\subset\{0,...,d_1\}$ then $P_J=P_1\cap\RR^d_J$, and if $J\subset\{d_1+1,...,d_1+d_2\}$ then $P_J=P_2\cap\RR^d_J$.
In all other cases with $|J|=2$, $P_J$ is a diamond.
\end{proof}

As seen in \cref{fig:4_gon_cases}, a minimizer is not necessarily a Hanner polytope in its~unconditional form, but might be a generalized Hanner polytope. 
It is convenient~to consider polytopes only up to a particular equivalence relation which identifies generalized Hanner polytopes with the usual Hanner polytopes.

%We can however~find normal form which will turn out to be always unconditional.

\subsection{Halfspace scaling}
\label{sec:halfspace_scaling}
Given $i\in[d]$ and $\alpha_+,\alpha_->0$, a \Def{halfspace scaling} is a piecewise linear map that sends $x\in\RR^d$ onto $y\in\RR^d$ with
$$y_i = \begin{cases}
    \alpha_+ x_i & \text{if $x_i>0$}
    \\
    \alpha_- x_i & \text{if $x_i\le 0$}
\end{cases}.$$
Halfspace scalings map proper locally anti-blocking polytopes onto proper locally anti-blocking polytopes and preserves their combinatorial type.
For this section we say that two polytopes are \Def{halfspace-scaling equivalent} if they are related by a sequence of halfspace scalings.

Note that a polytope is a generalized Hanner polytope if and only if it is~halfspace-scaling equivalent to a Hanner polytope (see also our definition of a \emph{normal form}~below).

% \subsection{Orthant scaling}

% Given real numbers $\alpha_i^\sigma>0$ for $i\in[d]$ and $\sigma\in \{+1,-1\}$, an \emph{orthant scaling} is a continuous transformation sending $x\in\RR^d$ onto $y\in\RR^d$ with
% %
% $$y_i = \begin{cases}
%     \alpha_i^{+1} x_i & \text{$\sign(x_i)=+1$}
%     \\
%     \alpha_i^{-1} x_i & \text{$\sign(x_i)=-1$}
% \end{cases}.$$
% %
% Orthant scalings map proper locally anti-blocking polytopes onto proper locally anti-blocking polytopes and preserves their combinatorial type.
% We say that two polytopes are \Def{equivalent} if they are related by orthant scaling.

% Given $i\in[d]$, $\sigma\in\{+,-\}$ and $\alpha>0$, a \Def{halfspace scaling} is continuous transfor\-mation $\phi_{i,\sigma,\alpha}\:\RR^d\to\RR^d$ that sends $x$ onto $y$ with
% %
% $$y_j = \begin{cases}
%     \alpha x_j & \text{if $j=i$ and $\sign(x_j)=\sigma$}
%     \\
%     \phantom\alpha x_j & \text{otherwise}
% \end{cases}.$$
% % 
% Proper locally anti-blocking polytopes $P$ are closed under halfspace scaling. We~say that $P$ is \Def{normalized} if $\pi_i(P)=\conv\{-e_i,e_i\}$ for each $i\in[d]$. $P$ can be normalized using a sequence of halfspace scaling.
% If $P$ is normalized, then so is its polar dual.

\subsection{Reconstructing $\boldsymbol{P}$ from $\boldsymbol{G_P}$.}

Our next goal is to reduce the classification of minimizers to a combinatorial problem. 
For that purpose, we show that $P$ can be reconstructed from $G_P$ up to equivalence.

For our discussion below it is sufficient and convenient to restrict to the positive orthant $\RR^d_+:=\{x\in\RR^d\mid x_i\ge 0\}$.
%For that purpose we introduce some useful notation: %: let $\RR^d_+:=\{x\in\RR^d\mid x_i\ge 0\}$ be the (open) positive orthant 
%and $\partial \RR^d_+:=\bigcup\{\RR^d_\sigma\mid\sigma \in\{0,+1\}^d\}$ its boundary. 
We introduce the following useful notation:\nls $\sigma_+:=(+1,...,+1)$, $F_+:=F_{\sigma_+}$ and $p_+:=p_{\sigma_+}$.
We also use the notation $F_+(P)$ and $p_+(P)$ to emphasize the implicit argument, which then allows us to write, for example, $F_+(P_J)$ or $F_+(\dual P)$ to refer to different setting if needed.
%Below most proofs are stated for $F_+$ only, but it is clear that they apply to, say, $F_+(P_J)$ etc.\ as well.
%More generally, for a polytope $Q\subset\RR^d_J$ with $\dim(Q)=|J|$ let $\sigma_+(Q)\in\{0,+1\}^d$ be the sign vector with $\supp(\sigma_+(Q))=J$.
%We then define $F_+(Q):=F_{\sigma_+(Q)}$ and $p_+(Q):=F_{\sigma_+(Q)}$. Note that $F_+(P)=F_+$, etc.

% Moreover, for a locally anti-blocking polytope $Q$ (for us this will be $Q=\dual P$ or $Q=P_J$) we write $F_+(Q)$ for its positive face, and $p_+(Q)$ accordingly.

The following technical properties of $F_+$ will be used extensively.

\iftrue % short version

\begin{proposition}
\label{res:technical_Fplus}
\label{res:technical_Fplus_properties}
\label{res:Fplus_properties}
For a face $F\subseteq P$ the following are equivalent:
\begin{myenumerate}
    \item $F=F_+$.
    \item $\relint(F)\subset\RR^d_+$.
    %\item $\relint(N_P(F))\subseteq\RR^d_+$.
    \item $\fdual F=F_+(\dual P)$, where $\fdual F\subseteq \dual P$ is the dual face to $F$.
\end{myenumerate}
    % \begin{myenumerate}
    %     \item $F=F_+$ if and only if $\relint(F)\subset\RR^d_+$.
    %     \item $F=F_+$ if and only if $N_P(F)\subset\RR^d_+$.
    % \end{myenumerate}
\end{proposition}
\begin{proof}
    \emph{\itmto 1 2:}
    Since $p_+\in \relint(F_+)$ and $\inv p_+\in N_P(F_+)$ are both in $\RR^d_+$, from \cref{res:point_on_both_sides} follows that all other inner points of $F_+$ must be in $\RR^d_+$ as well.

    \emph{\itmto 2 1:}
    Since $F=F_\sigma$ for some $\sigma\in\PMZ^d$, from $p_\sigma\in\relint(F_\sigma)\subset\RR^d_+$ follows $\sigma=\sigma_+$ and $F_\sigma=F_+$.
    
    \emph{\itmto 2 3:}
    %Suppose $F=F_+$. 
    We have $\fdual F=F_\sigma(\dual P)$ for some $\sigma\in\PMZ^d$. Then $\inv p_\sigma(\dual P)\in \relint(N_{\dual P}(\fdual F))=\cone(\relint(F))\subseteq\RR^d_+$. Hence $\sigma=\sigma_+$ and $\fdual F=F_+(\dual P)$.
    %The other direction is proven equivalently by swapping $P$ with $\dual P$.
    
    \emph{\itmto 3 1:} Swap $P$ and $\dual P$ (and $F$ and $\fdual F$) and apply \itmto 1 3.
\end{proof}

\else

\begin{lemma} 
\label{res:technical_Fplus}
\label{res:technical_Fplus_properties}
\label{res:Fplus_properties}
For a face $F\subseteq P$ the following are equivalent:
\begin{myenumerate}
    \item $F=F_+$.
    \item $\relint(F)\subset\RR^d_+$.
    \item $\relint(N_P(F))\subseteq\RR^d_+$.
    \item $\fdual F=F_+(\dual P)$, where $\fdual F\subseteq \dual P$ is the dual face to $F$.
\end{myenumerate}
    % \begin{myenumerate}
    %     \item $F=F_+$ if and only if $\relint(F)\subset\RR^d_+$.
    %     \item $F=F_+$ if and only if $N_P(F)\subset\RR^d_+$.
    % \end{myenumerate}
\end{lemma}
\begin{proof}
    We have $F=F_\sigma$ for some $\sigma\in\PMZ^d$.

    \emph{\itm1\,$\Rightarrow$\,\itm2:}
    Since $p_+\in \relint(F_+)$ and $\inv p_+\in N_P(F_+)$ are both in $\RR^d_+$, by \cref{res:point_on_both_sides} all other inner points of $F_+$ must be in $\RR^d_+$ as well.

    \emph{\itm2, \itm3\,$\Rightarrow$\,\itm1:}
    From either $p_\sigma\in\relint(F_\sigma)\subset\RR^d_\sigma$, or $\inv p_\sigma\in N_P(F_\sigma)\subseteq\RR^d_\sigma$, follows $\sigma=\sigma_+$ and $F_\sigma=F_+$.
    
    \emph{\itm1\,$\Leftrightarrow$\,\itm4:}
    suppose $F=F_+$. Since $\fdual F=F_\sigma(\dual P)$ for some $\sigma\in\PMZ^d$, we have $\inv p_\sigma(\dual P)\in \relint(N_{\dual P}(\fdual F))=\cone(\relint(F))\subseteq\RR^d_+$. This $\sigma=\sigma_+$ and $\fdual F=F_+(\dual P)$.

    \emph{\itm1\,$\Leftrightarrow$\,\itm4:}
    If $F=F_+$, then $\relint(N_{\dual P}(\fdual F))=\cone(\relint(F))\subseteq\RR^d_+$ where the inclusion uses \itmto 1 2.
    It follows $\fdual F=F_+(\dual P)$ by \itmto 3 1 applied to $\fdual F\subseteq\dual P$.
    The other direction is proven~equi\-valently by swapping $P$ with $\dual P$. % and $F$ with $\fdual F$.

    \emph{\itm1\,$\Rightarrow$\,\itm3:}
    We have $\fdual F=F_+(\dual P)$ by \itmto 1 4 and $\relint(F)\subset\relint(N_{\dual P}(\fdual F))$ $\subseteq\RR^d_+$ by \itmto 1 3 applied to $\fdual F\subseteq \dual P$. $F=F_+$ follows from \itmto 2 1.
\end{proof}

\fi

\begin{proposition}\label{res:Fplus_boundary}
    If $F\subset F_+$ is a proper face, then $F=F_+\cap\RR^d_J$ for some $J\subset[d]$.
    %The faces of $F_+$ are of the form $F_+\cap\RR^d_J$ for $J\subset[d]$.
    %$\partial F_+\subset \partial\RR^d_+$.
\end{proposition}
\begin{proof}
    %Since $F\not= F_+$ it contains a point $x\not\in \RR^d_+$. If it is not of the form $F_+\cap\RR^d_J$, then it contains a point $y\in\RR^d_+$. By \cref{res:point_on_both_sides} $N_P(F)\subset\{x_i=0\}$.
    %
    Since $F\subset F_+\subset \cl(\RR^d_+)$ by \cref{res:Fplus_properties} \itmto 1 2, either holds $\smash{F\subset\partial \RR^d_+}$ or $\smash{\relint(F)\subset\RR^d_+}$.
    By \cref{res:Fplus_properties} \itmto 2 1, the latter is possible only if~$F=F_+$.
    Thus, $F\subset\partial \RR^d_+$, and $F$ is of the claimed form.
    %
    % If $F$ is a face of $F_+$, then $F\subseteq F_+\subset \cl(\RR^d_+)$ by \cref{res:Fplus_properties} \itmto 1 2.
    % Then either $\smash{F\subset\partial \RR^d_+}$~or $\smash{\relint(F)\subset\RR^d_+}$.
    % By \cref{res:Fplus_properties} \itmto 2 1, the latter is possible only if $F=F_+$.
\end{proof}

\begin{proposition}\label{res:Fplus_facet_simplex}
    If $F_+$ is a facet, then it must be a simplex of the form $$F_+=\conv\{\alpha_1 e_1,...,\alpha_d e_d\}\quad\text{with $\alpha_1,...,\alpha_d>0$}.$$
\end{proposition}
\begin{proof}
    By \cref{res:Fplus_boundary}, the faces of $F_+$ are of the form $F_+\cap\RR^d_J$ with $J\subseteq[d]$.
    This can produce at most $2^d$ faces (including the empty face). 
    Since $F_+$ is of~dimen\-sion $d-1$, it must be a simplex with exactly $2^d$ faces.
    Counting the faces in each dimension shows that the $\delta$-dimensional faces of $F_+$ are of the form $F_+\cap\RR^d_J$ with $|J|=\delta+1$.
    In particular, the vertices of $F_+$ are in $\RR^d_J=\Span\{e_i\}$ where $J=\{i\}$, that is, they are of the form $\alpha_i e_i$. 
    The claim follows.
\end{proof}

%$$\RR^d_{J+}:=\{x\in\RR^d\mid \text{$x_i>0$ if $i\in J$ and $x_i=0$ otherwise}\}.$$

\begin{proposition}
\label{res:one_facat_many_facts}
\label{res:one_vertex_many_vertices}
For $d\ge 3$ holds
\begin{myenumerate}
    \item 
    $F_+$ is a facet of $P$ if and only if $F_+(P_J)$ is a facet of $P_J$ for all $J\subset[d]$. 
    \item 
    $F_+$ is a vertex of $P$ if and only if $F_+(P_J)$ is a vertex of $P_J$ for all $J\subset[d]$.   
\end{myenumerate}
\end{proposition}
\begin{proof}
    % By \cref{res:Fplus_boundary}, the faces of $F_+$ are of the form $F\cap\RR^d_J$ for $J\subseteq[d]$.
    % This can produce at most $2^d$ faces. 
    % Thus, if $F_+$ is a facet (of dimension $d-1$) it must have exactly $2^d$ faces and is a simplex.

    \emph{\itm1\,$\Rightarrow$:}
    If $F_+$ is a facet, then it is of the form $\conv\{\alpha_1e_1,...,\alpha_d e_d\}$ by \cref{res:Fplus_facet_simplex}.
    Then for each $J\subseteq[d]$ the face $F_+\cap \RR^d_J=\conv\{\alpha_i e_i\mid i\in J\}$ is of dimension $|J|-1$, hence a facet of $P_J$ with $\relint(F_+\cap \RR^d_J)\subset\RR^d_{+J}:=\cone\{e_i\mid i\in J\}$.
    By \cref{res:Fplus_properties} \itmto 2 1 this facet is $F_+(P_J)$.
    % Assume $F_+$ is a facet.
    % By \cref{res:Fplus_boundary}, the faces of $F_+$ are of the form $F_J:=F_+\cap\RR^d_J$ for $J\subseteq[d]$.
    % This can produce at most $2^d$ faces. 
    % Since $F_+$ is of dimension $d-1$, it must have exactly $2^d$ faces and is a simplex.
    % % In fact, $F_J$ is of dimen\-sion $|J|-1$ and is therefore a facet of $P_J$.
    % In particular, if $|J|=\{i\}$, then $F_J\subset\Span\{e_i\}$ is a vertex of $F_+$ and must be of the form $\alpha_i e_i$ for some $\alpha_i>0$.
    % For general $J\subseteq[d]$, $F_J$ is therefore of the form $\conv\{\alpha_i e_i\mid i\in J\}$. 
    % This implies $\relint(F_J)\subset\RR^d_{J+}$ and $F_+(P_J)=F_J$ via \cref{res:relint_F_Fplus}.
    % $F_+(P_J)$ is therefore a facet of $P_J$.

    \emph{\itm2\,$\Rightarrow$:}
    This is \itm1\,$\Rightarrow$ applied to $\dual P$ and using \cref{res:Fplus_properties} \itmeq 1 3.

    \emph{\itm1\,$\Leftarrow$:}
    Assume $F_+(P_J)$ is a facet for all $J\subset[d]$.
    Suppose first that $F_+$ is a vertex. 
    By \itm2\,$\Rightarrow$, $F_+(P_J)$ is a vertex of $P_J$, in particular for $J:=\{1,2\}\subset [d]$ (here we use $d\ge 3$). Since in dimension two vertices and facets are distinct, this contradicts $F_+(P_J)$ being a facet. Thus $F_+$ cannot be vertex.
    
    Let then $x$ be a vertex of $F_+$, which is a proper face. 
    By \cref{res:Fplus_boundary}, $x\in\RR^d_\sigma$ for some $\sigma\in\{0,+1\}^d\setminus\{\sigma_+\}$. Set $J:=\supp(\sigma)$. 
    %By \cref{res:Fplus_boundary} $x\in\RR^d_\sigma$ for some $\sigma\in\{0,+1\}^d$ with $J:=\supp(\sigma)\subset[d]$. 
    Since $x$ is also a vertex of $P_J$, by \cref{res:Fplus_properties} \itmto 2 1 we have $F_+(P_J)=x$. 
    By assumption, $F_+(P_J)$ is a facet of $P_J$.
    Therefore, $P_J$ must be 1-dimensional and $J=\{i\}$.
    Thus, all vertices of $F_+$ are of the form $\alpha_i e_i$.
    Moreover, if $F_+$ does not have a vertex $\alpha_i e_i$, then $F_+\subset\{x_i=0\}$, contradicting $\relint(F_+)\subset\RR^d_+$ from \cref{res:Fplus_properties} \itmto 1 2. We conclude that $F_+=\conv\{\alpha_1 e_1,..., \alpha_d e_d\}$, and hence is a facet of $P$.

    \emph{\itm2\,$\Leftarrow$:}
    This is \itm1\,$\Leftarrow$ applied to $\dual P$ and using \cref{res:Fplus_properties} \itmeq 1 3.
\end{proof}

\begin{proposition}\label{res:vertex_if_complete}
    $F_+$ is a vertex if and only if $G_P$ is a complete graph.
\end{proposition}
\begin{proof}
    We proceed by induction on the dimension $d$ of $P$.
    If $d\in\{0,1\}$, then both $F_+$ is a vertex and $G_P$ is complete and the statement is satisfied.
    
    If $d=2$, then $F_+$ is a vertex if and only if $P$ is axis-aligned, which, by definition of $G_P$, is equivalent to $G_P$ being complete on two vertices.

    If $d\ge 3$, then by \cref{res:one_vertex_many_vertices} \itm2 $F_+$ is a vertex of $P$ if and only if $F_+(P_J)$~is~a vertex of $P_J$ for all $J\subset[d]$. By induction hypothesis, this is equivalent to all~$J\subset[d]$ inducing complete subgraphs of $G_P$. This in turn is equivalent to $G_P$ itself being complete.
\end{proof}

% Locally anti-blocking polytopes are closed under \emph{halfspace scaling}: given $\sigma\in\{+1,1\}$ and $\alpha>0$ then
% %
% $$x_i\mapsto \begin{cases}
%     \alpha x_i & \text{if $\sign(x_i)=\sigma$}
%     \\
%     \phantom\alpha x_i & \text{if $\sign(x_i)\not=\sigma$}
% \end{cases}.$$
% % $$y_i\mapsto \begin{cases}
% %     \alpha_i^+ x_i & \text{if $x\ge 0$}
% %     \\
% %     \alpha_i^- x_i & \text{if $x < 0$}
% % \end{cases}.$$
% %
% We say that two locally anti-blocking polytopes are \emph{equivalent} if they are related by a sequence 

We say that $P$ is \Def{normalized} if $\pi_i(P)=\conv\{-e_i,e_i\}$ for all $i\in[d]$.
Note~that each proper locally anti-blocking polytope is halfspace-scaling equivalent to a normalized one.

\begin{proposition}\label{res:exact_vertex_location}
    If $P$ is normalized and $F_+$ is a vertex, then $F_+=e_1+\cdots+e_d$.
\end{proposition}
\begin{proof}
    From \cref{res:Fplus_properties} \itmto 1 3 we know that $F_+(\dual P)$ is a facet of $\dual P$, which by \cref{res:Fplus_facet_simplex} is of the form $\conv\{\alpha_1 e_1,...,\alpha_d e_d\}$, where $\alpha_i e_i$ are vertices of $P$.
    Since $P$ is~normalized, we have $\alpha_i=1$, and $F_+(\dual P)=\conv\{e_1,...,e_d\}$. The dual face to this facet is the vertex $e_1+\cdots+e_d$, which is $F_+$ by \cref{res:Fplus_properties} \itmto 3 1.
\end{proof}

\begin{lemma}\label{res:unique_reconstruction}
    $P$ is halfspace-scaling equivalent to
    %is normalized, then it is uniquely determined by $G_P$ and equals
    %
    \begin{equation*}\label{eq:determinde_by_cliques}
    %\mathsf{UP}_{G_P}^\circ
    C(G_P):=\conv\Big\{\sum_{i\in J} \pm e_i \mid \text{$J\subseteq[d]$ induces a complete subgraph of $G_P$}\Big\}.
    \end{equation*}
\end{lemma}

Note that $C(G)$ is exactly the \emph{stable set polytope} $\mathsf UP_{\bar G}$ of the complement graph $\overline G$ % %(that is, defined on independent sets rather than complete subgraphs) 
as defined in \cite{KohlOlsenSanyal}.

%This construction of a polytope from a graph is related to the \emph{stable set polytope} defined on independent sets instead of cliques.
%In fact, $C(G)$ is exactly $\mathsf UP_{\bar G}$ defined in\cite{KohlOlsenSanyal}, where $\bar G$ denotes the complement of $G$.

\begin{proof}[Proof of \cref{res:unique_reconstruction}]
W.l.o.g.\ assume that $P$ is normalized.

Fix a vertex $x\in P$, and w.l.o.g.\ assume $\sign(x)\in\{0,+1\}^d$. Then $x$ is also~a~ver\-tex of $P_J$ where $J:=\supp(x)$ and we have $F_+(P_J)=x$ by \cref{res:Fplus_properties} \itmto 2 1.
Hence $J$ induces a complete graph of $G_P$ by \cref{res:vertex_if_complete}, and we have
$x=\sum_{i\in J} e_i$ by \cref{res:exact_vertex_location}.
We conclude $P\subseteq C(G_P)$.

Conversely, assume that $J\subseteq[d]$ induces a complete subgraph of $G_P$. Then~by \cref{res:vertex_if_complete} and \cref{res:exact_vertex_location} we have that $F_+(P_J)=\sum_{i\in J} e_i$ is a vertex of~$P_J$, and hence is contained in $P$.
Clearly, all arguments apply analogously to orthants other than $\RR^d_+$, and so $\sum_{i\in J}\pm e_i$ is a vertex of $P_J$ for each choice of signs.
Thus, we conclude $C(G_P)\subseteq P$.
\end{proof}

We have now reduced the task of classifying the minimizers to a combinatorial investigation of $G_P$.

\subsection{Characterizing Hanner polytopes}

A graph is a \Def{cograph} if
\begin{myenumerate}
    \item it is a single vertex.
    \item it is the complement of a cograph.
    \item it is the disjoint union of two cographs.
\end{myenumerate}

\begin{lemma}\label{res:charcterizing_Hanner}
    $P$ is halfspace-scaling equivalent to a Hanner polytope if and only if $G_P$ is a cograph.
\end{lemma}

\begin{proof}
    Assume first that $P\subseteq\RR^d$ is a Hanner polytope.
    We show that $G_P$ is a~co\-graph by induction on the recursive construction of $P$.
    If $d=1$ then $G_P$ is a single vertex, hence a cograph.
    If $d\ge 2$ then $P$ has been constructed in one of two ways:
    \begin{myenumerate}
        \item $P$ is the polar dual of a previously constructed Hanner polytope $Q$. Hence $G_{Q}$ is a cograph by induction hypothesis. Then $G_P$ is the complement of $G_{Q}$ (\cf\ \cref{res:GP_facts} \itm1), and therefore a cograph itself.
        \item $P=P_1\oplus P_2$ with previously constructed Hanner polytopes $P_i$. 
        The $G_{P_i}$ are cographs by induction hypothesis.
        In conclusion, $G_P=G_{P_1}\uplus G_{P_2}$,\nls (\cf\ \cref{res:GP_facts} \itm3) and is a cograph itself.
    \end{myenumerate}
    
    This shows that $G_P$ is a cograph whenever $P$ is a Hanner polytope.
    Conversely, since clearly all cographs can be obtained in this way, and since the polytope $P$ is uniquely determined from $G_P$ via \cref{res:unique_reconstruction} (up to halfspace-scaling equivalence), $P$ is halfspace-scaling equivalent to a Hanner polytope whenever $G_P$ is a cograph.
\end{proof}

It is known that cographs can be characterized in the following way:

\begin{lemma}[\!{\cite[Theorem 2]{corneil1981complement}}]
    \label{res:cograph_P3}
    A graph $G$ is a cograph if and only if it contains no path of length three (on four vertices) as an induced subgraph.
\end{lemma}

Let $\Pi_3$ be the path of length three with vertices $\{1,2,3,4\}$ and edges $\{12,23,34\}$.

\begin{lemma}\label{res:P3_not_minimal}
    $C(\Pi_3)$ is not a minimizer.
\end{lemma}

\begin{proof}
    Observe that the face $F\subset C(\Pi_3)$ with normal vector $(1,1,1,1)$ is spanned by the vertices $e_1+e_2,e_2+e_3$ and $e_3+e_4$. In particular, $\relint(F)\subset\RR^d_+$.
    Thus, if $C(\Pi_3)$ were a minimizer, then $F=F_+$ using \cref{res:Fplus_properties} \itmto 2 1. 
    But~the edge $\conv\{e_1+e_2,e_3+e_4\}$ of $F$ is not of the form $F_+\cap\RR^d_J$, contradicting \cref{res:Fplus_boundary}.
\end{proof}

Alternatively, that $C(\Pi_3)$ is not a minimizer can be checked directly by counting its faces (\eg\ using Sage).
Its 12 vertices have coordinates
$$
(\pm1,\pm1,0,0),\quad
(0,\pm1,\pm1,0),\quad
(0,0,\pm1,\pm1)
$$
(corresponding to the cliques of $\Pi_3$).
We find that $C(\Pi_3)$ has $f$-vector $(12,36,36,12)$ and exactly $97>3^4$ non-empty faces.
%This polytope also appears in \cite{freij2012face} as~a~member~of~a family of Hansen polytopes with exactly $3^d+16$ non-empty faces, here for $d=4$.

With this in place we can prove the minimizer part of \cref{thm:main_lab}: if $P\subset\RR^d$~is proper locally anti-blocking with $s(P)=3^d$, then $P$ is a generalized Hanner poly\-tope:

% \begin{theorem}
%     If $P\subset\RR^d$ is proper locally anti-blocking with $s(P)=3^d$, then $P$ is combinatorially equivalent to a Hanner polytope.
% \end{theorem}

\begin{proof}[Proof of \cref{thm:main_lab} (the minimizers)]
    By \cref{res:3_to_J} $P\cap\RR^d_J$ is a minimizer for all $J\subseteq[d]$. By \cref{res:P3_not_minimal} $P\cap\RR^d_J$ cannot be (equivalent to) $C(\Pi_3)$. Thus, by \cref{res:GP_facts} \itm2, no $J\subseteq[d]$ induces a $\Pi_3$-subgraph of $G_P$, and  $G_P$ is therefore a cograph by \cref{res:cograph_P3}. In conclusion, $P$ is equivalent to a Hanner polytope by \cref{res:charcterizing_Hanner}.
\end{proof}

\tempnewpage

%\section{A combinatorial proof for unconditional polytopes}
\section{A combinatorial proof of \cref{cor:main_unconditional}}
\label{sec:uncondition}

In this section we present a second independent proof of the $3^d$ conjecture specifically for unconditional polytopes that is more combinatorial in nature. % and pertains to the structure of face lattices.

\iftrue % relatively complemented

The following proof of \cref{cor:main_unconditional} is based on the fact that the face lattice~$\F(P)$ of a polytope is \emph{relatively complemented}.
A bounded lattice $\mathcal L$ with least element $\hat 0$ and greatest element~$\hat 1$ is \Def{complemented}, if each $F\in\mathcal L$ has a (not necessarily~unique) \Def{complement} $G\in\mathcal L$, \ie\
$$F\land G= \hat 0 \wideand F\lor G=\hat 1,$$
where $\land$ and $\lor$ denote the meet and join in $\mathcal L$.
It is \Def{relatively complemented} if every interval $[F,G]$ is complemented.

For the face lattice $\F(P)$ we have $\hat 0=\eset$ and $\hat 1=P$. 
It is well-known that face lattices are relatively complemented, but we include a short argument:

\begin{lemma}\label{res:relatively_complemented}
    The face lattice $\F(P)$ of a polytope is relatively complemented
\end{lemma}
\begin{proof}
Since intervals in face lattices are face lattices themselves, it suffices to show that $\F(P)$ is complemented. 
%Recall that the least and greatest element in $\F(P)$~are $\eset$ and $P$ respectively.

Clearly, $\eset$ and $P$ are complements of each other. 
If $F$ is a non-empty proper face, then project $P$ ``along'' $F$, that is, so that $F$ becomes a vertex $v$ of~the projection. 
In the projection the vertex $v$ has a complement, namely, any facet that does not contain the vertex.
The preimage of this facet under the projection is a complement of the original face $F$.
\end{proof}

\else

The following proof of \cref{cor:main_unconditional} is based on the fact that the face lattice~$\F(P)$ of a polytope is always \Def{complemented}, that is, every face $F\in\mathcal F(P)$ has a (not~necessarily unique) \Def{complement} $G\in\mathcal F(P)$, \ie\
$$F\land G=\eset \wideand F\lor G=P.$$
where $\land$ and $\lor$ denote the meet and join in $\F(P)$.
This is well-known, but we give a short argument: $\eset$ and $P$ are clearly complements of each other. 
If $F$ is non-empty and proper, then project $P$ ``along'' $F$, that is, so that $F$ becomes a vertex $v$ of~the projection. 
In the projection the vertex $v$ has a complement, namely, any facet that does not contain the vertex.
The preimage of this facet under the projection is a complement of the original face $F$.

\fi

% \begin{lemma}
%     The face lattice $\F(P)$ is \Def{relatively complemented}, \ie\ every interval $[F,G]$ is complemented.
% \end{lemma}
% %
% \begin{proof}
%     \TODO
% \end{proof}

Note also that in a centrally-symmetric polytope for each non-empty proper face $F\in\F(P)$ the face $-F$ is a complement.

% \begin{theorem}
%     If $P\subset\RR^d$ is unconditional, then $s(P)\ge 3^d$.
% \end{theorem}
% \label{thm:unconditional}

\begin{proof}[Proof of \cref{cor:main_unconditional} (the inequality)]
    We proceed by induction on the dimension~$d$~of $P$. The cases $d\in\{0,1\}$ are trivial.
    For the induction step fix a coordinate~hyper\-plane $H$ and a proper face $F\in\F(P)$ supported by a translate of $H$.
    The following three sets $S_+,S_0,S_-\subset\F(P)$ are disjoint:
    \begin{align*}
        S_+ &:=\{\text{faces of $P$ that intersect $F$ but not $-F$}\},
        \\
        \mathrlap{S_0}\phantom{S_+} &:=\{\text{faces of $P$ that intersect both of $F$ and $-F$, or neither}\},
        \\
        S_- &:=\{\text{faces of $P$ that intersect $-F$ but not $F$}\}.
    \end{align*}
    We show that each set contains $\ge3^{d-1}$ faces, which proves the claim.

    \ul{Case $S_0$}: the faces of $P\cap H$ correspond to faces of $P$ that are $H$-symmetric (\ie\ invariant under reflection in $H$). Each $H$-symmetric face is necessarily in $S_0$. Since $P\cap H$ is unconditional of dimension $d-1$, we have $|S_0|\ge s(P\cap H)\ge 3^{d-1}$~by~in\-duction hypothesis.

    \ul{Case $S_+$/$S_-$}: by symmetry it suffices to consider $S_+$. 
    For each $\check F\in(\eset, F]$ and $\hat F\in[F,P)$ the face $ F$ has a complement $G(\check F,\hat F)$ in $[\check F,\hat F]$ (using \cref{res:relatively_complemented}), that is, $ G(\check F,\hat F)\land F=\check F$ and $ G(\check F,\hat F)\lor F=\hat F$.
    In particular, the $ G(\check F,\hat F)$ are distinct since they~differ~in~meet and join with $ F$. Below we will show that $ G(\check F,\hat F)\in S_+$. 
    Assuming this claim,\nls and observing that both $F$ and the dual face $ F^\diamond\in\F(P^\circ)$ are unconditional, the induction hypothesis applies and it already follows
    \begin{align*}
        |S_+|&\ge |\{ G(\check F,\hat F): \eset \subset\check F\subseteq F\subseteq\hat F\subset P\}| 
        \\&= s\big((\eset, F]\big)\cdot s\big([ F,P)\big) = s( F)\cdot s(F^\diamond) \ge 3^{\dim( F)}\cdot 3^{d-1-\dim( F)} =3^{d-1},
    \end{align*}
    
    It remains to show that indeed $ G(\check F,\hat F)\in S_+$.
    By construction $ G(\check F,\hat F)$ intersects $ F$ in $\check F\not=\eset$. It remains to show that $ G(\check F,\hat F)$ does \emph{not} intersect $- F$.
    But indeed, if $G':= G(\check F,\hat F)\land (- F)\not=\eset$, then both $ G'\subseteq G(\check F,\hat F)\subseteq\hat F$ and $ G'\subseteq- F\Longrightarrow- G'\subseteq F\subseteq\hat F$, hence $ G'\lor(- G')\subseteq \hat F\subset P$, in contradiction to the fact that $ G'$ and $- G'$ are complements in $P$ and $ G'\lor(- G')=P$.
\end{proof}

% It appears possible to generalize the proof to unconditional polytopes and~to~potentially extract the minimizers from it, though we did not pursuit this given our more general proof of \cref{thm:main_lab}.

\section{Open questions}

%We presented two different proofs (one geometric, one combinatorial) for Kalai's $3^d$ conjecture in the class of locally antiblocking polytopes.
%We furthermore verified that in this class the minimizers are exactly the Hanner polytopes.

Below we collected some remaining questions and related conjectures.

\subsection{Combinatorially locally anti-blocking}

Since Kalai's $3^d$ conjecture makes a statement only about the polytope's combinatorics, our results apply to not only locally anti-blocking polytope, but to any polytope that is combinatorially equivalent to a locally anti-blocking polytope.
This raises~the question: are there~centrally-sym\-metric polytopes that are not combinatorially equivalent to a locally anti-block\-ing polytope?
While the answer is most certainly \emph{yes}, we are not aware of any~example.

%We want to draw attention to the following peculiarity: 
% Being locally anti-blocking is a~geometric property, while Kalai's $3^d$ conjecture makes a statement about the polytope's combinatorics. 
% As a consequence, for our results to apply to a polytope $P$ it needs not be locally anti-blocking, but it suffices to be \emph{combinatorially equivalent} to a locally anti-blocking polytope (\eg\ the right-most polygon in \cref{fig:4_gon_cases}).

\begin{challenge}
    Construct (many) centrally-symmetric polytopes that are \ul{not} combinatorially equivalent to a locally anti-blocking polytope.
\end{challenge}

A hope is that by understanding how to construct such polytopes, one will also understand why they might need to have many faces.

\subsection{The flag conjecture}

The following conjecture was brought to our attention by Matthias Schymura and might be amenable to techniques similar to the ones we employed in this article.
Recall that a \emph{flag} of $P$ is an inclusion chain $\eset=F_{-1}\subseteq F_0$ $\subseteq F_1\subseteq \cdots\subseteq F_d=P$, where $F_i$ is an $i$-dimensional face of $P$. 
%The following analogous conjecture for the number of flags has~been~proposed:

\begin{conjecture}%[Flag conjecture]
\label{conj:flag}
    Every $d$-dimensional centrally-symmetric polytope $P$ has at least as many flags as the $d$-dimensional cube, \ie\
    $$\text{\normalfont \#flags}(P) \ \ge \ \text{\normalfont \#flags}(C_d) \ = \ 2^d d!\,.$$
\end{conjecture}

Like Kalai's $3^d$ conjecture, \cref{conj:flag} is widely open, and it appers reasonable to first attack it for the class of locally anti-blocking polytopes.

\tempnewpage